\documentclass[12pt]{amsart}
\usepackage{amsmath,amsfonts,amssymb}
\numberwithin{equation}{section}

\newtheorem{theorem}{Theorem}[section]

\newtheorem{proposition}[theorem]{Proposition}

\newtheorem{definition}[theorem]{Definition}

 \DeclareMathOperator{\Tr}{Tr}
\DeclareMathOperator{\Ad}{Ad}

\DeclareMathOperator{\End}{End} \DeclareMathOperator{\Rea}{Re}
\DeclareMathOperator{\Ima}{Im}

\title [Schr\"odinger model...]{Schr\"odinger model and Stratonovich-Weyl correspondence for Heisenberg motion groups}
\author {Benjamin Cahen}
\address{Universit\'e de Lorraine, Site de Metz, UFR-MIM,
D\'epartement de math\'ematiques,
B\^atiment A,
Ile du Saulcy, CS 50128, F-57045, Metz cedex 01, France.}
\email{benjamin.cahen@univ-lorraine.fr}

\subjclass[2000]{22E45; 22E70; 22E20; 81S10; 81R30.} \keywords{
Stratonovich-Weyl correspondence; Berezin quantization; Berezin transform; Heisenberg motion group; 
reproducing kernel Hilbert space; coherent states; Schr\"odinger representation;
Bargmann-Fock representation; Segal-Bargmann transform.}

\begin{document}

\maketitle

\begin{abstract}
We introduce a Schr\"odinger model for the unitary irreducible representations of a Heisenberg motion group and we show that the usual Weyl quantization then provides a Stratonovich-Weyl correspondence.
\end{abstract}

\vspace{1cm}

\section {Introduction} \label{sec:intro}

There are different ways to extend the usual Weyl correspondence between functions on  ${\mathbb R}^{2n}$ and operators
on $L^2( {\mathbb R}^{n})$ to the general setting of a Lie group acting on a homogeneous space \cite{AE}, \cite{Go}, \cite{CaWQ}, \cite{GazH}. Here we are concerned with Stratonovich-Weyl correspondences. The notion of Stratonovich-Weyl correspondence was introduced in 
\cite{St} and its systematic study began with the work of  J.M. Gracia-Bond\`{i}a, J.C.
V\`{a}rilly and their co-workers (see \cite{GBV}, \cite{FGBV},
\cite{CGBV}, \cite{GB} and also \cite{BM}). The following definition is taken from \cite{GB}, see also \cite{GBV}.

\begin{definition} Let $G$ be a Lie group and
$\pi$ be a unitary representation of $G$ on a Hilbert space $\mathcal
H$. Let $M$ be a homogeneous $G$-space and let $\mu$ be
a $G$-invariant measure on $M$. Then a
Stratonovich-Weyl correspondence for the triple $(G,\pi, M)$ is an
isomorphism ${\mathcal W}$ from a vector space of operators on $\mathcal H$ to
a vector space of functions on $M$ satisfying the following
properties:

 \begin{enumerate}

\item the function ${\mathcal W}(A^{\ast})$ is the complex-conjugate of ${\mathcal W}(A)$;

\item Covariance: we have ${\mathcal W}(\pi (g)\,A\,\pi (g)^{-1})(x)={\mathcal W}(A)(g^{-1}\cdot x)$;

\item Traciality: we have
\begin{equation*}\int_M\,{\mathcal W}(A)(x){\mathcal W}(B)(x)\,d\mu (x)=\Tr(AB).
\end{equation*}

\end{enumerate} \end{definition}

Stratonovich-Weyl
correspondences were constructed for various Lie group representations, see  \cite{CGBV}, \cite{GB}.
In particular, in \cite{CaPad}, Stratonovich-Weyl correspondences for the holomorphic
representations of quasi-Hermitian Lie groups were obtained by
taking the isometric part in the polar decomposition of the Berezin
quantization map, see also \cite{FGBV}, \cite{CaSWC}, \cite{CaSWD}, \cite{AU1}, \cite{AU2} and \cite{CaJac}.

The basic example is the case when $G$ is the $(2n+1)$-dimensional Heisenberg group
acting on ${\mathbb R}^{2n}\cong {\mathbb C}^n$ by translations. 
Each non-degenerate unitary irreducible representation of
$G$ has then two classical realizations: the Schr\"odinger model
on $L^2({\mathbb R}^n)$ and the Bargmann-Fock model on the
Fock space \cite{Fo}, an intertwining operator between these
realizations being the Segal-Bargmann transform \cite{Fo}, \cite{Comb}.
In this context, it is well-known that the usual
Weyl correspondence provides a Stratonovich-Weyl correspondence for
the Schr\"odinger realization \cite{ArnC}, \cite{Wild1}, \cite{Pe}. It is also known that
this Stratonovich-Weyl correspondence is connected by the
Segal-Bargmann transform to the Stratonovich-Weyl correspondence
for the Bargmann-Fock realization which was obtained by polarization
of the Berezin quantization map \cite{Luo1}, \cite{Luo}.
In \cite{CaRiv}, we obtained similar results for the $(2n+2)$-dimensional real diamond group. This group, also called
oscillator group,  is a semidirect product of the Heisenberg group by the real
line. 

The aim of the present paper is to extend the preceding results to the Heisenberg motion groups. An Heisenberg motion group is the semidirect product of the $(2n+1)$-dimensional Heisenberg group $H_n$ by a compact subgroup $K$ of the unitary group $U(n)$. Note that Heisenberg motion groups play an important role in the theory of Gelfand pairs, since the study of a Gelfand pair of the form $(K_0,N)$ where $K_0$ is a compact Lie group acting by automorphisms on a nilpotent Lie group $N$ can be reduced
to that of the form $(K_0,H_n)$, see \cite{BJLR}, \cite{BJR}.

More precisely, we introduce a Schr\"odinger realization for the unitary irreducible representations of a 
Heisenberg motion group and we prove that we obtain a Stratonovich-Weyl correspondence by combining the usual Weyl correspondence and the unitary part
of the Berezin calculus for $K$.

Let us briefly describe our construction. First notice that each Heisenberg motion group is, in particular, a quasi-Hermitian Lie group and  that we can obtain its unitary irreducible representations as holomorphically induced representations on some generalized Fock space by the general method of \cite{Ne}, Chapter XII. Then we can get
Schr\"odinger realizations for these representations by using, as in the case of the Heisenberg group, a
generalized Bargmann-Fock transform. Hence we obtain a Stratonovich-Weyl correspondence for such a Schr\"odinger realization by introducing a generalization of the usual Weyl correspondence.

Note that, in \cite{Mo}, a Schr\"odinger model and a generalized Segal-Bargmann transform for the scalar highest weight representations of an Hermitian Lie group of tube type were introduced and studied. Let us also mentioned
that B. Hall has obtained some generalized Segal-Bargmann transforms in various situations by means of the heat kernel, see \cite{Hall} and references therein. Then one can hope for futher generalizations of our results to quasi-Hermitian Lie groups. 

This paper is organized as follows. In Section \ref{sec:Heisen}, we review some well-known facts about the Fock model
and the Schr\"odinger model of the unitary irreducible representations of an Heisenberg group and about the corresponding
Berezin calculus and Weyl correspondence. In Section \ref{sec:gen}, we introduce the Heisenberg motion groups and,
in Section \ref{sec:Fock} and Section \ref{sec:swcb}, we describe their unitary irreducible representations in the Fock model and the associated
Berezin calculus. We introduce the (generalized) Segal-Bargmann transform and the Schr\"odinger model in Section \ref{sec:Sch}. In Section \ref{sec:SWC}, we show that the usual Weyl correspondence also gives a Stratonovich-Weyl correspondence for the Schr\"odinger model. Moreover, we compare it with the Stratonovich-Weyl correspondence for
the Fock model which is directly obtained by polarization of the Berezin quantization map.

\section {Heisenberg groups} \label{sec:Heisen}

In this section, we review some well-known results about the
the Schr\"odinger model and the Fock
model of the unitary irreducible (non-degenerated)
representations of the Heisenberg group. We follow the presentation
of \cite{CaRiv} in a large extend.

Let $G_0$ be the Heisenberg group of dimension $2n+1$ and
${\mathfrak g}_0$ be the Lie algebra of $G_0$. Let
$\{X_1,\ldots,X_n,Y_1,\ldots,Y_n,\tilde {Z}\}$ be a basis of
${\mathfrak g}_0$ in which the only non trivial brackets are
$[X_k\,,\,Y_k]=\tilde {Z}$, $ 1\leq k\leq n$ and let
$\{X_1^{\ast},\ldots,X_n^{\ast},Y_1^{\ast},\ldots,Y_n^{\ast},{\tilde
Z}^{\ast}\}$ be the corresponding dual basis of ${\mathfrak
g}_0^{\ast}$.

For $a=(a_1,a_2,\ldots,a_n)\in {\mathbb R}^n$, $b=(b_1,b_2,\ldots,b_n)\in
{\mathbb R}^n$ and $c\in {\mathbb R}$, we denote by $[a,b,c]$ the
element $\exp_{G_0}(\sum_{k=1}^na_kX_k+\sum_{k=1}^nb_kY_k+c{\tilde
Z})$ of $G_0$. Similarly, for $\alpha=(\alpha_1,\alpha_2,\ldots,\alpha_n)\in {\mathbb R}^n$, $\beta=(\beta_1,\beta_2,\ldots,\beta_n)\in
{\mathbb R}^n$ and $\gamma\in {\mathbb R}$, we denote by $(\alpha,\beta,\gamma)$ the element
$\sum_{k=1}^n{\alpha}_kX_k^{\ast}+\sum_{k=1}^n{\beta}_kY_k^{\ast}+
{\gamma}{\tilde Z}^{\ast}$ of ${\mathfrak
g}_0^{\ast}$. The coadjoint action of $G_0$ is then given by
\begin{equation*}\Ad ^{\ast}([a,b,c])\, (\alpha,\beta,\gamma)=(\alpha+\gamma\beta, \beta-\gamma\alpha,\gamma). \end{equation*}

Now we fix a real number $\lambda>0$ and denote by ${\mathcal O}_{\lambda}$ the
orbit of the element $\lambda {\tilde Z}^{\ast}$ of
${\mathfrak g}_0^{\ast}$ under the coadjoint action of $G_0$ (the
case $\lambda<0$ can be treated similarly). By the Stone-von Neumann
theorem, there exists a unique (up to unitary equivalence) unitary
irreducible representation of $G_0$ whose restriction to the center
of $G_0$ is the character $[0,0,c]\rightarrow e^{i\lambda c}$
\cite{AKo}, \cite{Fo}. Note that this representation is associated with the
coadjoint orbit ${\mathcal O}_{\lambda}$ by the Kirillov-Kostant
method of orbits \cite{Kir}, \cite{Ko}. More precisely, if we choose the real
polarization at $\lambda {\tilde Z}^{\ast}$ to be the space spanned by the
elements $Y_k$ for $ 1\leq k\leq n$ and $\tilde {Z}$ then we obtain
the Schr\"odinger representation $\sigma_0$ realized on
$L^2({\mathbb R}^n)$ as
\begin{equation*}(
\sigma_0([a,b,c])f)(x)
=e^{i\lambda(c-bx+\frac{1}{2}ab)}f(x-a), \end{equation*} see
\cite{Fo} for instance. Here we denote
$xy:=\sum_{k=1}^nx_ky_k$ for $x=(x_1,x_2,\ldots,x_n)$ and
$y=(y_1,y_2,\ldots,y_n)$ in ${\mathbb R}^n$. 

The differential of
$\sigma_0$ is then given by
\begin{equation*}d\sigma_0(X_k)f(x)=-\partial_k f(x),\,\,
d\sigma_0(Y_k)f(x)=-i\lambda
x_kf(x),\,\,d\sigma_0(\tilde {Z})f(x)=i\lambda
f(x)\end{equation*} where $k=1,2,\ldots, n$. 

On the other hand, if we consider the complex polarization at $\lambda {\tilde Z}^{\ast}$ to be the space spanned
by the elements $X_k+iY_k$ for $ 1\leq k\leq n$ and $\tilde {Z}$
then the method of orbits leads to the Bargmann-Fock representation $\pi_0$
defined as follows \cite{CaJAM}.

Let ${\mathcal F}_{0}$ be the Hilbert space of holomorphic functions $F$
on ${\mathbb C}^n$ such that \begin{equation*}\Vert F\Vert^2_{{\mathcal F}_{0}}
:=\int_{{\mathbb C}^n} \vert F(z)\vert^2\, e^{-\vert
z\vert^2/2\lambda}\,d\mu_{\lambda} (z) <+\infty\end{equation*} where
 $d\mu_{\lambda}(z):=(2\pi
\lambda)^{-n}\,dx\,dy$. Here $z=x+iy$ with $x$ and $y$ in ${\mathbb
R}^n$.

Let us consider the action of $G_0$ on ${\mathbb C}^n$
defined by $g\cdot z:=z+\lambda (b-ia)$ for $g=[a,b,c]\in G_0$ and
$z\in {\mathbb C}^n$. Then $\pi_{0}$ is the representation of
$G_0$ on ${\mathcal F}_{0}$ given by
\begin{equation*}\pi_0 (g)\,F(z)=\alpha
(g^{-1},z)\,F(g^{-1}\cdot z) \end{equation*} where the map $\alpha$
is defined by
\begin{equation*}\alpha (g,z):=\exp
\bigl(-ic\lambda+(1/4) (b+ai)(-2z+\lambda(-b+ai))\bigr)
\end{equation*} for $g=[a,b,c]\in G_0$ and $z\in {\mathbb C}^n$.

 The differential of $\pi_{0}$ is then given by
\begin{equation*}\left\{\begin{aligned}
d\pi_{0}(X_k)F(z)=&\frac{1}{2}iz_kF(z)+\lambda i\frac
{\partial F}{\partial z_k}\\
d\pi_{0}(Y_k)F(z)=&\frac{1}{2}z_kF(z)-\lambda \frac {\partial
F}{\partial z_k}\\
d\pi_{0}(\tilde {Z})F(z)=&i\lambda F(z).
\end{aligned}\right.\end{equation*}

As in \cite{Hall1}, Section 6 or
\cite{Comb}, Section 1.3, we can verify by using the previous formulas for $d\pi_0$ and $d\sigma_0$ that the Segal-Bargmann
transform $B_0: L^2({\mathbb R}^n)\rightarrow {\mathcal F}_{0}$
defined by
\begin{equation*}B_0(f)(z)=(\lambda/ \pi)^{n/4}\,\int_{{\mathbb R}^n}\,e^{
(1/4\lambda)z^2+ixz-(\lambda/2)x^2}\,f(x)\,dx\end{equation*} is a
(unitary) intertwining operator between $\sigma_{0}$ and
$\pi_{0}$. The inverse Segal-Bargmann transform
$B_0^{-1}=B_0^{\ast}$ is then given by
\begin{equation*}B_0^{-1}(F)(x)=(\lambda/ \pi)^{n/4}\,\int_{{\mathbb C}^n}\,e^{
(1/4\lambda){\bar z}^2-ix{\bar z}-(\lambda/2)x^2}\,F(z)\,e^{-\vert
z\vert^2/2\lambda}\,d\mu_{\lambda} (z).\end{equation*}

For $z\in {\mathbb C}^n$, consider the coherent state
$e_z(w)=\exp ({\bar z}w/2\lambda)$. Then we have the reproducing property
$F(z)=\langle F,e_z\rangle_{{\mathcal F}_{0}}$ for each $F\in {\mathcal
F}_{0}$ where $\langle\cdot ,\cdot\rangle_{{\mathcal
F}_{0}}$ denotes the scalar product on ${\mathcal F}_{0}$.

Now, we introduce the Berezin quantization map and we
review some of its properties. Let ${\mathcal C}_0$ be the space of all operators (not necessarily
bounded) $A$ on ${\mathcal F}_{0}$ whose domain contains $e_z$
for each $z\in {\mathbb C}^n$. Then the Berezin symbol of $A\in {\mathcal C}_0$ is the
function $S^0(A)$ defined on ${\mathbb C}^n$ by
\begin{equation*}S^0(A)(z):=\frac {\langle A\,e_z\,,\,e_z\rangle_{\mathcal F_0}}
 {\langle e_z\,,\,e
_z\rangle_{\mathcal F_0}}. \end{equation*}

We have the following result, see for instance \cite{CaRiv}.

\begin{proposition} \label{proprBer}\begin{enumerate}
\item Each $A\in {\mathcal C}_0$ is determined by $S^0(A)$;

\item For each $A\in {\mathcal C}_0$ and each $z\in {\mathbb C}^n$, we
have $S^0(A^{\ast})(z)=\overline {S^0(A)(z)}$;
 
\item For each $z\in {\mathbb C}^n$, we have $S^0(I_{{\mathcal F}_{0}})(z)=1$.
Here $I_{{\mathcal F}_{0}}$ denotes the identity operator of ${\mathcal F}_0$;

\item For each $A\in {\mathcal C}_0$, $g\in G_0$ and $z\in {\mathbb C}^n$, we have
$\pi_0(g)^{-1}A\pi_0(g)\in {\mathcal C}_0$ and
\begin{equation*}S^0(A)(g\cdot z)=S^0(\pi_0(g)^{-1}A\pi_0(g))(z);\end{equation*}

\item The map $S^0$ is a bounded operator from
${\mathcal L}_2({\mathcal F}_{0})$ (endowed with the
Hilbert-Schmidt norm) to $L^2({\mathbb C}^n,\mu_{\lambda})$  which
is one-to-one and has dense range.

\end{enumerate} \end{proposition}

\begin{proof} For (1) and (2), see \cite{Be1} and \cite{CGR}. Note that (4) follows from the following property:
 For each $g\in G_0$ and each
$z\in {\mathbb C}^n$, we have $\pi_{0}(g)e_z=\overline {\alpha
(g,z)}e_{g\cdot z}$, see \cite{CaPad}. Finally, (5) is a particular case of \cite{UU}, Proposition 1.19.
\end{proof}

Recall that the Berezin transform is then the operator ${\mathcal B}^0$
on $L^2({\mathbb C}^n,\mu_{\lambda})$ defined by ${\mathcal B}^0=S^0(S^0)^{\ast}$. Thus
we have the integral formula
\begin{equation*}{\mathcal B}^0(F)(z)=\int_{{\mathbb C}^n}\,F(w)
\,e^{ \vert z-w\vert^2/2\lambda}\,d\mu_{\lambda}(w),\end{equation*}
see \cite{Be1}, \cite{Be2}, \cite{UU}, \cite{OZ} for instance.
Recall also that we have ${\mathcal B}^0=\exp (\lambda\Delta/2)$ where
$\Delta=4\sum_{k=1}^n\partial^2/\partial z_k\partial {\bar z}_k$,
see \cite{UU}, \cite{Luo}. 

Note that Berezin transforms have been studied, in the general setting, by many authors,
see in particular \cite{UU}, \cite{No}, \cite{DOZ}, \cite{OZ} and \cite{Z}.

Note also that $S^0$ allows us to connect $\pi_0$ to ${\mathcal O}_{\lambda}$
as shown by the following proposition. Here we denote by ${\mathfrak
g}_0^c$ the complexification of ${\mathfrak g}_0$.

\begin{proposition} \cite{CaRiv} Let $\Phi_{\lambda}$ be the map defined by
\begin{equation*}\Phi_{\lambda} (z):=\sum_{k=1}^n(\Rea z_k
X_k^{\ast}+\Ima z_kY_k^{\ast})+\lambda{\tilde
Z}^{\ast}.\end{equation*} \noindent Then
\begin{enumerate} \item For each $X\in {\mathfrak g}_0^c$ and each
$z\in {\mathbb C}^n$, we have
\begin{equation*}S^0(d\pi_{0}(X))(z)=i \langle \Phi_{\lambda}
(z),X\rangle.
 \end{equation*}

\item For each $g\in G_0$ and each $z\in {\mathbb C}^n$, we have
$ \Phi_{\lambda}(g\cdot z)=\Ad^{\ast}(g)\,\Phi_{\lambda}(z)$.

\item The map $\Phi_{\lambda}$ is a diffeomorphism from ${\mathbb C}^n$ onto
${\mathcal O}_{\lambda}$. \end{enumerate} \end{proposition}

 Now we aim to transfer $S^0$ to operators on $L^2({\mathbb R}^n)$. To this goal, we define  $S^1(A):=S^0(B_0AB_0^{-1})$
for $A$ operator on $L^2({\mathbb R}^n)$. Of course, the properties of $S^0$ give rise to similar properties of
$S^1$. In particular, $S^1$ is a bounded operator from ${\mathcal
L}_2(L^2({\mathbb R}^n))$ to $L^2({\mathbb C}^n,\mu_{\lambda})$
and $S^1$ is $G_0$-covariant with respect to
$\sigma_{0}$.

Moreover, denoting by $I_{B_0}$ the (unitary) map from ${\mathcal L}_2(L^2({\mathbb
R}^n))$ onto ${\mathcal L}_2({\mathcal F}_{0})$ defined by
$I_{B_0}(A)=B_0AB_0^{-1}$, we have $S^1=S^0I_{B_0}$ then
\begin{equation*}S^1(S^1)^{\ast}=(S^0I_{B_0})(S^0I_{B_0})^{\ast}=S^0I_{B_0}I_{B_0}^{\ast}(S^0)^{\ast}=S^0(S^0)^{\ast}=
{\mathcal B}^0.\end{equation*} This shows that the Berezin transform
corresponding to $S^1$ is the same as the Berezin transform
corresponding to $S^0$. 
Then we can write the polar decompositions
of $S^0$ and $S^1$ as $S^0=({\mathcal B}^0)^{1/2}U^0$ and
$S^1=({\mathcal B}^0)^{1/2}U^1$ where the maps $U^0:{\mathcal L}_2({\mathcal
F}_{0})\rightarrow L^2({\mathbb C}^n,\mu_{\lambda})$ and
$U^1:{\mathcal L}_2(L^2({\mathbb R}^n))\rightarrow L^2({\mathbb
C}^n,\mu_{\lambda})$ are unitary. 

Moreover,
as in the proof of \cite{CaSWD}, Proposition 3.1, we can verify that $U^0$
is a Stratonovich-Weyl correspondence for
$(G_0,\pi_{0},{\mathbb C}^n)$ and that $U^1$ is a
Stratonovich-Weyl correspondence for $(G_0,\sigma_{0},{\mathbb
C}^n)$. Note that  $G_0$-covariance of $U^0$ and
$U^1$ immediately follows from $G_0$-covariance of $S^0$ and
$S^1$. Note also that we have $U^1=U ^0I_{B_0}$.

Now, we show how to use the usual Weyl correspondence in order to get another
Stratonovich-Weyl correspondence for $\sigma_{0}$.
The Weyl correspondence on ${\mathbb R}^{2n}$ is defined as follows.
For each $f$ in the Schwartz space ${\mathcal S}({\mathbb
R}^{2n})$, let  $ W_0(f)$ be the operator on  $L^2({\mathbb R}^{n})$ defined by \begin{equation*}W_0(f)\phi
(p)={(2\pi)}^{-n}\,\int_{{\mathbb R}^{2n}}\, e^{isq} \,f( p+(1/2)s,
q)\,\phi (p+s)\,ds\,dq.
\end{equation*}

The Weyl calculus can be extended to much larger
classes of symbols (see for instance \cite{Ho}). In particular, if $ f(p,q)=u(p)q^{\alpha}$ where
$u\in C^{\infty}({\mathbb R}^n)$ then we have,  see \cite{Vo},
\begin{equation*}W_0(f)\phi (p)=\left(i\frac {\partial }{ \partial
s}\right)^{\alpha} \left( u(p+(1/2)s)\,\phi (p+s) \right) \Bigl
\vert _{s=0}.\end{equation*}  

From this, we can deduce the following proposition. Consider the action of $G_0$ on ${\mathbb R}^{2n}$
given by $g\cdot (p,q):=(p+a,q+\lambda b)$ where $g=[a,b,c]$. 

\begin{proposition} \cite{CaRiv} Let $\Psi_{\lambda}$ be the map defined by
\begin{equation*}\Psi_{\lambda}
(p,q):=\sum_{k=1}^n(q_k X_k^{\ast}-\lambda
p_kY_k^{\ast})+\lambda{\tilde Z}^{\ast}.\end{equation*} \noindent
Then \begin{enumerate} \item For each $X\in {\mathfrak g}_0^c$ and
each $(p,q)\in {\mathbb R}^{2n}$, we have
\begin{equation*}W_0^{-1}(d\sigma_{0}(X))(p,q)=i \langle \Psi_{\lambda} (p,q),
X\rangle.\end{equation*}

\item For each $g\in G_0$ and each $(p,q)\in {\mathbb R}^{2n}$, we have
$ \Psi_{\lambda}(g\cdot (p,q))=\Ad^{\ast}(g)\,\Psi_{\lambda}(p,q)$.

\item The map $\Psi_{\lambda}$ is a diffeomorphism from ${\mathbb R}^{2n}$ onto
${\mathcal O}_{\lambda}$.

\item For
each $(p,q)\in {\mathbb R}^{2n}$, we have  $\Phi_{\lambda}(q-\lambda pi)=\Psi_{\lambda}(p,q)$. 
\end{enumerate} \end{proposition}

Assume that ${\mathbb R}^{2n}$ is equipped with the $G_0$-invariant measure $\tilde{\mu}:=(2\pi)^{-n}dpdq$.
Then one has the following result.
\begin{proposition} \cite{Fo}, \cite{CaRiv} The map $W_0^{-1}$ is a Stratonovich-Weyl correspondence
for $(G_0, \sigma_{0},{\mathbb R}^{2n})$.
\end{proposition}

The following proposition asserts that if we identify ${\mathbb
R}^{2n}$ with ${\mathbb C}^n$ by the map $j:(p,q)\rightarrow q-\lambda pi$
then the unitary part in the polar decomposition of $S^1$ coincides
with the inverse of the Weyl transform, see \cite{Luo} and \cite{OZ}. 

\begin{proposition} \label{propLuo} Let $J$ be the map from $ L^2({\mathbb C}^n,\mu_{\lambda})$ onto $L^2({\mathbb
R}^{2n})$  defined by $J(F)=F\circ j$.  Then we have $U^1=(W_0J)^{-1}$.\end{proposition}

Finally, note that we can obtain Stratonovich-Weyl correspondences for $(G_0, \sigma_{0},{\mathcal
O}_{\lambda})$ and $(G_0,
\pi_{0},{\mathcal O}_{\lambda})$ by transferring $W_0^{-1}$ and $U^0$
by using $\Phi_{\lambda}$ and $\Psi_{\lambda}$.
More precisely, let $\nu_{\lambda}$ be the $G_0$-invariant measure on
${\mathcal O}_{\lambda}$ defined by
$\nu_{\lambda}:=(\Phi_{\lambda}^{-1})^{\ast}(\mu_{\lambda})=(\Psi_{\lambda}^{-1})^{\ast}({\tilde
\mu})$. Then the maps $\tau_{\Phi_{\lambda}}:F\rightarrow F\circ
\Phi_{\lambda}^{-1}$ from $L^2({\mathbb C}^n,\mu_{\lambda})$ onto
$L^2({\mathcal O}_{\lambda},\nu_{\lambda})$ and
$\tau_{\Psi_{\lambda}}:F\rightarrow F\circ \Psi_{\lambda}^{-1}$ from
$L^2({\mathbb R}^{2n})$ onto $L^2({\mathcal
O}_{\lambda},\nu_{\lambda})$ are unitary and we have
$\tau_{\Phi_{\lambda}}=\tau_{\Psi_{\lambda}} J$. Hence we can assert
the following proposition.

\begin{proposition} \label{propSwc} The map ${\mathcal W}_1:=\tau_{\Psi_{\lambda}}W_0^{-1}$ is a
Stratonovich-Weyl correspondence for $(G_0, \sigma_{0},{\mathcal
O}_{\lambda})$, the map ${\mathcal W}_2:=\tau_{\Phi_{\lambda}}U^0$
is a Stratonovich-Weyl correspondence for $(G_0,
\pi_{0},{\mathcal O}_{\lambda})$ and we have ${\mathcal
W}_1={\mathcal W}_2I_{B_0}$.
\end{proposition}

\section {Generalities on Heisenberg motion groups} \label{sec:gen}

In order to introduce the Heisenberg motion groups, it is convenient to write the elements
of the Heisenberg group $G_0$ and its multiplication law as follows. For each $z\in {\mathbb C}^n$, $c\in {\mathbb R}$,
we denote here by $(z,{\bar z}, c)$ the element $G_0$ which is denoted by $[\Rea z, \Ima z, c]$ in Section \ref{sec:Heisen}. Moreover, for each $z,\,w \in {\mathbb C}^{n}$, we denote $zw:=\sum_{k=1}^nz_kw_k$ and  we consider the symplectic form $\omega$ on ${\mathbb C}^{2n}$ defined by
\begin{equation*}\omega
((z,w),(z',w'))=\frac{i}{2}(zw'-z'w).
\end{equation*} for $z,w,z',w'\in {\mathbb C}^n$.
Then the multiplication of $G_0$ is given by
\begin{equation}\label{multHeisen}((z,{\bar
z}),c)\cdot ((z',{\bar z'}),c')=((z+z',{\bar z}+{\bar
z'}),c+c'+\tfrac{1}{2}\omega ((z,{\bar z}),(z',{\bar
z'}))),\end{equation}
the complexification $G_0^c$ of $G_0$ is
$G_0^c=\{((z,w),c)\,:z,w \in {\mathbb C}^n, c\in {\mathbb C}\}$ and
the multiplication of $G_0^c$ is obtained by replacing $(z,{\bar z})$
by $(z,w)$ and $(z',{\bar z'})$ by $(z',w')$ in Eq. \ref{multHeisen}. 

Now, let $K$ be a closed subgroup of $U(n)$. Then $K$ acts on $G_0$ by $k\cdot ((z,{\bar
z}),c)=((kz,\bar
{kz}),c)$ and we can form the semidirect product $G:=G_0\rtimes K$ which is called a Heisenberg motion
group. The elements of $G$ can be written as $((z,{\bar
z}),c,k)$ where $z\in {\mathbb C}^n$, $c\in {\mathbb R}$ and $k\in
K$. The multiplication of $G$ is then given by
\begin{equation*}((z,{\bar
z}),c,k)\cdot ((z',{\bar z'}),c',k')=((z,{\bar z})+(kz',\bar
{kz'}),c+c'+\tfrac{1}{2}\omega ((z,{\bar z}),(kz',\bar
{kz'})),kk').\end{equation*}

We denote by $K^c$ the complexification of $K$ and we consider the action of $K^c$ on ${\mathbb C}^n\times
{\mathbb C}^n$ given by $k\cdot (z,w)=(kz, (k^t)^{-1}w)$ (here, the subscript $t$ denotes transposition).
The group $G^c$ is then
the semidirect product $G^c=G_0^c\rtimes K^c$. The elements of $G^c$ can be written as $((z,w),c,k)$ where $z,\,w\in {\mathbb C}^n$, $c\in {\mathbb C}$ and $k\in
K^c$ and the multiplication law of $G^c$ is given by
\begin{equation*}((z,w),c,k)\cdot ((z',w'),c',k')=((z,w)+k\cdot (z',w'),c+c'+\tfrac{1}{2}\omega ((z,w),k\cdot (z',w')),kk').\end{equation*}

We denote by
${\mathfrak k}$, ${\mathfrak k}^c$, ${\mathfrak g}$ and ${\mathfrak
g}^c$ the Lie algebras of $K$, $K^c$, $G$ and $G^c$. The derived action of ${\mathfrak k}^c$ on 
${\mathbb C}^n\times {\mathbb C}^n$ is then $A\cdot (z,w):=(Az,-A^tw)$ and
the Lie
brackets of ${\mathfrak g}^c$ are given by
\begin{equation*}[((z,w),c,A),((z',w'),c',A')]=(A\cdot (z',w')-A'\cdot(z,w),
\omega ((z,w),(z',w')),[A,A']).\end{equation*}

Let ${\tilde K}$ be the subgroup of $G$ defined by 
${\tilde K}:=\{((0,0),c,k)\,:\,c\in {\mathbb R},\,k\in K\}$. Also, let ${\mathfrak
h}_0$ be a Cartan subalgebra of ${\mathfrak
k}$. Then  the Lie algebra $\tilde {\mathfrak k}$ of ${\tilde K}$ is a 
maximal compactly embedded subalgebra of $\mathfrak g$ and the subalgebra $\mathfrak h$ of $\mathfrak g$
consisting of all elements of the form $((0,0),c,A)$ where $c\in {\mathbb R}$ and $A\in {\mathfrak
h}_0$ is a compactly embedded Cartan subalgebra of $\mathfrak g$ \cite{Ne}, p. 250.

Following \cite{Ne}, Chapter XII.1, we set ${\mathfrak p}^+=\{((z,0),0,0)\,:\,z\in {\mathbb C}^n\}$ and ${\mathfrak p}^-=\{((0,w),0,0)\,:\,w\in {\mathbb C}^n\}$ and we denote by $P^+$ and  $P^-$ the corresponding analytic subgroups
of $G^c$, that is, 
$P^+=\{((z,0),0,I_n)\,:\,z\in {\mathbb C}^n\}$ and $P^-=\{((0,w),0,I_n)\,:\,w\in {\mathbb C}^n\}$.

Note that $G$ is a group of the Harish-Chandra type
\cite{Ne}, p. 507 (see also \cite{Sa} and \cite{He}, Chapter VIII), that is, the following properties are satisfied:
\begin{enumerate}\item ${\mathfrak g}^c={\mathfrak p}^+\oplus {\tilde {\mathfrak
k}}^c \oplus {\mathfrak p}^-$ is a direct sum of vector spaces,
$({\mathfrak p}^+)^{\ast}={\mathfrak p}^-$ and $[\tilde {\mathfrak
k}^c,{\mathfrak p}^{\pm}]\subset {\mathfrak p}^{\pm}$;
\item The multiplication map $P^+{\tilde K}^cP^-\rightarrow G^c$,
$(z,k,y)\rightarrow zky$ is a biholomorphic diffeomorphism onto its
open image;
\item $G\subset P^+{\tilde K}^cP^-$ and $G\cap {\tilde K}^cP^-={\tilde K}$. \end{enumerate}

We denote by $p_{{\mathfrak p}^+}$, $p_{{\tilde {\mathfrak k}}^c}$ and $p_{{\mathfrak p}^-}$
the projections of ${\mathfrak g}^c$ onto ${\mathfrak p}^+$,
${\tilde {\mathfrak k}}^c$ and ${\mathfrak p}^-$ associated with the above direct
decomposition.

We can easily verify that each $g=((z_0,w_0), c_0,k)\in G^c$ has a $P^+{\tilde K}^cP^-$-decomposition given by
\begin{equation*} g=((z_0,0),0,I_n)\cdot ((0,0),c,k)\cdot ((0,w_0),0,I_n)\end{equation*}
where $c=c_0-\tfrac{i}{4}z_0w_0$. We denote by
$\zeta:\, P^+{\tilde K}^cP^-\rightarrow P^+$, $\kappa:\, P^+{\tilde K}^cP^-\rightarrow
K^c$  and $\eta:\, P^+{\tilde K}^cP^-\rightarrow P^-$ the projections onto
$P^+$-, ${\tilde K}^c$- and $P^-$-components.

We can introduce an action (defined almost everywhere) of $G$ on ${\mathfrak p}^+$ as follows.
For $Z\in {\mathfrak p}^+$ and
$g\in G^c$, we define 
$g\cdot Z \in {\mathfrak p}^+$ by $g\cdot Z:=\log \zeta (g\exp Z)$.
From the above formula for the $P^+{\tilde K}^cP^-$-decomposition, we deduce that
if $g=((z_0,w_0), c_0,k)\in G$
and $Z=((z,0),0,0)\in {\mathfrak p}^+$ then we have $g\cdot Z=\log  \zeta (g\exp Z)=((z_0+kz,0),0,0)$.
Note that  $\mathcal D:=G\cdot 0={\mathfrak p}^+\simeq {\mathbb C}^n$ here.

A useful section $Z\rightarrow g_Z$ for the action of $G$ on $\mathcal D$ can be obtained by using Proposition 4.5
of \cite{CaPar}. Here we get $g_Z=((z,{\bar z}),0,I_n)$ for each $Z=((z,0),0,0)$, $z\in {\mathbb C}^n$.

Now we compute the adjoint and coadjoint actions of $G^c$. 
Let $g=(v_0,c_0,k_0)\in G^c$ where $v_0\in {\mathbb C}^{2n}$,
$c_0\in {\mathbb C}$, $k_0\in K^c$ and
$X=(w,c,A)\in {\mathfrak g}^c$ where $w\in {\mathbb C}^{2n}$, $c\in
{\mathbb C}$ and $A\in {\mathfrak k}^c$. We can easily verify that
\begin{align*}\Ad&(g)X=\frac{d}{dt}(g\exp(tX)g^{-1})\vert_{t=0}\\
&=\bigl( k_0w-(\Ad(k_0)A)\cdot v_0, c+\omega (v_0,k_0w)-\tfrac{1}{2}\omega
(v_0,(\Ad(k_0)A)\cdot v_0),\Ad(k_0)A\bigr).\end{align*}

Now, let us denote by $\xi=(u,d,\phi)$, where $u\in {\mathbb C}^{2n}$,
$d\in {\mathbb C}$ and $\phi \in {({\mathfrak k}^c)}^{\ast}$, the
element of $({\mathfrak g}^c)^{\ast}$ defined by
\begin{equation*}\langle \xi,(w,c,A)\rangle =\omega (u,w)+dc+\langle
\phi, A\rangle. \end{equation*} Also, for $u,v\in {\mathbb
C}^{2n}$, we denote by $v\times u$ the element of $({{\mathfrak
k}^c})^{\ast}$ defined by $\langle v\times u, A\rangle:=\omega
(u,A\cdot v)$ for $A\in {\mathfrak k}^c$.
Then, from the above formula for the adjoint action, we deduce  that for each $\xi=(u,d,\phi)\in ({\mathfrak g}^c)^{\ast}$ and
$g=(v_0,c_0,k_0)\in G^c$ we have
\begin{equation*} \Ad^{\ast}(g)\xi= \bigl(k_0u-dv_0,d,\Ad^{\ast}(k_0)\phi
+v_0\times (k_0u-\tfrac {d}{2}v_0)\bigr)\end{equation*} By
restriction, we also get the analogous formula for the coadjoint action of
$G$. From this, we see that if a coadjoint orbit of $G$ contains a point $(u,d,\phi)$
with $d\not= 0$ then it also contains a point of the form $(0,d,\phi_0)$. Such an orbit is called
{\it generic}.

\section {Fock model  for Heisenberg motion groups} \label{sec:Fock}

In this section, we introduce the Fock model of the unitary irreducible representations of $G$
by using the general method of \cite{Ne}, Chapter XII that we describe here briefly.

Let $\rho$ be a unitary irreducible representation of $K$ on a (finite-dimensional) Hilbert space $V$
and $\lambda \in {\mathbb R}$. Let $\tilde \rho$ be the representation of $\tilde K$ on $V$ defined by ${\tilde \rho} ((0,0),c, k)=e^{i\lambda c}\rho (k)$ for each $c\in {\mathbb R}$ and $k\in K$.

For each $Z, W \in {\mathcal D}$, let $K(Z,W):={\tilde \rho}
(\kappa (\exp W^{\ast}\exp Z))^{-1}$ and for each $g\in G$, $Z\in {\mathcal D}$, let $J(g,Z):=\tilde{ \rho} (\kappa (g\exp Z))$, \cite{Ne}, Chapter XII.1. Consider the Hilbert space $\tilde {\mathcal F}$ of all holomorphic functions on $\mathcal D$ with
values in $V$ such that \begin{equation*}\Vert f\Vert^2_{\tilde {\mathcal F}}
:=\int_{\mathcal D}\,\langle K(Z,Z)^{-1} f(Z),f(Z)\rangle_V\,d\mu
(Z)<+\infty
\end{equation*}
where $\mu$ denotes an invariant $G$-measure on $\mathcal D$. Then the equation 
\begin{equation*}{\tilde \pi} (g)f(Z)=J(g^{-1},Z)^{-1}\,f(g^{-1}\cdot Z)\end{equation*}
defines a unitary representation of $G$ on $\tilde {\mathcal F}$.
This representation can be also obtained by holomorphic induction from $\tilde {\rho}$,
that is, it corresponds to the natural action of $G$ on the square-integrable holomorphic sections of the Hilbert $G$-bundle $G\times_{\tilde{\rho}}V$ over $G/K\cong \mathcal D$ \cite{CaRiv}. Note also that $\tilde {\pi}$ is irreducible since $\tilde {\rho}$ is irreducible, \cite{Ne}, p. 515.

Here we can easily compute $K$ and $J$. For each $Z=((z,0),0,0),\,W=((w,0),0,0)\in {\mathcal D}$, we have $K(Z,W)=e^{\lambda z{\bar w}/2}I_V$ and for each $g=((z_0,{\bar z_0}), c_0,k)\in G$ and $Z=((z,0),0,0),\in {\mathcal D}$, we have
\begin{equation*}J(g,Z)=\exp \left(i\lambda c_0+\tfrac{\lambda}{2}{\bar z_0}(kz)+\tfrac{\lambda}{4}\vert z_0\vert^2\right)\,\rho(k).\end{equation*} 

Moreover, $\mu$ can be taken to be the $G$-invariant measure on
${\mathcal D}\simeq {\mathbb C}^n$ defined by
 $d\mu (Z):={\lambda}^n({2\pi})^{-n}\,dx\,dy$. Here $Z=((z,0),0,0)$ and $z=x+iy$ with $x$ and $y$ in ${\mathbb R}^n$. From now on, we identify $Z=((z,0),0,0)\in {\mathcal D}$ with $z\in {\mathbb C}^n$ and each function on
${\mathcal D}$ with the corresponding function on ${\mathbb C}^n$.

Consequently, the Hilbert product on
$\tilde{\mathcal F}$ is given by
\begin{equation*}\langle f,g\rangle_{\tilde{\mathcal F}}
=\int_{{\mathbb C}^n} \langle f(z),g(z)\rangle_V\, e^{-\lambda\vert
z\vert^2/2}\, \left(\frac{\lambda}{2\pi}\right)^n\,dx\,dy\end{equation*}
and we get the following formula for $\tilde{\pi}$:
\begin{equation*}(\tilde{\pi}(g)f)(z)=\exp \left(i\lambda c_0+\tfrac{\lambda}{2}{\bar z_0}z-\tfrac{\lambda}{4}\vert z_0\vert^2\right)\,\rho(k)\,f(k^{-1}(z-z_0)) \end{equation*}
where $g=((z_0,{\bar z_0}), c_0,k)\in G$ and $z \in {\mathbb C}^n$.

In fact, in order to use the results of Section \ref{sec:Heisen}, it is convenient to replace $\tilde{\pi}$ by
an equivalent representation $\pi$ whose restriction to $G_0$ is precisely $\pi_0$. To this aim, we consider
the Fock space $\mathcal F$ of all holomorphic functions $f: {\mathbb C}^n\rightarrow V$ such that
\begin{equation*}\Vert f\Vert_{\mathcal F}^2
:=\int_{{\mathbb C}^n} \Vert f(z)\Vert_V^2\, e^{-\vert
z\vert^2/2\lambda}\, d\mu_{\lambda}(z)< +\infty.\end{equation*}

Let ${\mathcal J}:{\tilde{\mathcal F}}\rightarrow \mathcal F$ be the unitary operator defined by ${\mathcal J}(f)(z)=f(i{\lambda}^{-1}z)$
and set $\pi(g):={\mathcal J}\tilde{\pi}(g){\mathcal J}^{-1}$ for each $g\in G$. Then we have

\begin{equation*}({\pi}(g)f)(z)=\exp \left(i\lambda c_0+\tfrac{1}{2}i{\bar z_0}z-\tfrac{\lambda}{4}\vert z_0\vert^2\right)\,\rho(k)\,f(k^{-1}(z+i\lambda z_0)) \end{equation*}
where $g=((z_0,{\bar z_0}), c_0,k)\in G$ and $z\in {\mathbb C}^n$.

We can easily compute the differential of $\pi$: 

\begin{proposition}\label{redpi}  Let $X=((a,{\bar a}),c,A)\in {\mathfrak g}$. Then, for each $f\in {\mathcal F}$ and each $z\in {\mathbb C}^n$, we have
\begin{equation*}(d\pi (X)f)(z)=d\rho(A)f(z)+i(\lambda c+\tfrac{1}{2}{\bar a}z)f(z)+df_z(-Az+i\lambda a).\end{equation*}
\end{proposition}

Clearly, one has ${\mathcal F}={\mathcal F}_0\otimes V$. For $f_0\in {\mathcal F}_0$ and $v\in V$,
we denote by $f_0\otimes v$ the function $z\rightarrow f_0(z)v$. Moreover, if $A_0$ is an operator of 
${\mathcal F}_0$ and $A_1$ is an operator of $V$ then we denote by $A_0\otimes A_1$ the operator
of ${\mathcal F}$ defined by $(A_0\otimes A_1)(f_0\otimes v)=A_0f_0\otimes A_1v$.

Let $\tau$ be the left-regular representation of $K$ on 
${\mathcal F}_0$, that is, $(\tau (k)f_0)(z)=f_0(k^{-1}z)$. Then we have 
\begin{equation}\label{eq:decomppi} \pi((z_0,{\bar z_0}),c_0,k)=\pi_0((z_0,{\bar z_0}),c_0) \tau(k)\otimes \rho(k)\end{equation} for each $z_0\in {\mathbb C}^n$,
$c_0\in {\mathbb R}$ and $k\in K$. Note that this is precisely Formula (3.18) in \cite{BJLR}.

\section {Stratonovich-Weyl correspondence via Berezin quantization} \label{sec:swcb}

In this section, we introduce the Berezin quantization map associated with $\pi$ and the corresponding
Stratonovich-Weyl correspondence. We consider first the Berezin quantization map associated with $\rho$
\cite{ACGc}, \cite{CaS}, \cite{Wild2}.

Let us fix a positive root system of $\mathfrak k$ relative to ${\mathfrak h}_0$ and denote by $\Lambda \in 
({\mathfrak h}_0^c)^{\ast}$ the highest weight of $\rho$ and by ${\mathfrak k}^c= {\mathfrak n}^+ \oplus {\mathfrak h}_0^c \oplus
{\mathfrak n}^-$ the corresponding triangular decomposition of ${\mathfrak k}^c$. Let $\tilde{\varphi_0}$ be
the element of $({\mathfrak k}^c)^{\ast}$ defined by $\tilde{\varphi_0}=-i\Lambda$ on ${\mathfrak h}_0$ and
by $\tilde{\varphi_0}=0$ on ${\mathfrak n}^{\pm}$. We denote by $\varphi_0$ the restriction of $\tilde{\varphi_0}$
to $\mathfrak k$. Then the orbit  $o(\varphi_0)$ of $\varphi_0$ under the coadjoint action of $K$
is said to be associated with $\rho$ \cite{CaWQ}, \cite{Wild2}.

Here we assume that $\varphi_0$ is regular in the sense that the stabilizer of $\varphi_0$
for the coadjoint action of $K$ is precisely the connected subgroup $H_0$ of $K$ with Lie
algebra ${\mathfrak h}_0$ \cite{CaS}.

Note that a complex structure on $o(\varphi_0)$ is then defined
by the diffeomorphism $o(\varphi_0)\simeq K/H_0 \simeq K^c/H_0^cN^-$
where $H_0$ is the connected subgroup of $K$ with Lie algebra  ${\mathfrak h}_0$ and $N^-$ is the analytic subgroup of $K^c$ with Lie algebra
${\mathfrak n}^-$.

Without loss of generality, we can assume that $V$ is a space
of holomorphic sections of a complex line bundle over $o(\varphi_0)$ as in \cite{CaS}.
Let $\varphi \in o(\varphi_0)$. For each $\hat{\varphi}\not= 0$ in the fiber over $\varphi$,
there exists a unique function $e_{\hat{\varphi}}\in V$ (a coherent
state) such that $a(\varphi)=\langle a,e_{\hat{\varphi}} \rangle_Ve_{\hat{\varphi}}$ for
each $a\in V$.

The Berezin calculus on $o(\varphi_0)$ associates
with each operator $B$ on $V$ the complex-valued function $s(B)$ on
$o(\varphi_0)$ defined by
\begin{equation*}s(B)(\varphi)=\frac {\langle Be_{\hat{\varphi}},e_{\hat{\varphi}} \rangle_V}
{\langle e_{\hat{\varphi}},e_{\hat{\varphi}} \rangle_V} \end{equation*}
which is
called the symbol of $B$. This definition makes sense since the right side of the equation does not depend on $\hat{\varphi}$ in the fiber over ${\varphi}$ but only on $\varphi$. We denote by $Sy(o(\varphi_0))$ the space of all such symbols.
Then we have the following proposition, see \cite{CGR}, \cite{ACGc} and \cite{CaS}.

\begin{proposition} \label{propripetit}
\begin{enumerate}
\item The map $B\rightarrow s(B)$ is injective.

\item  For each operator $B$ on $V$, we have $s(B^{\ast})=\overline
{s(B)}$.

\item  For each $\varphi \in o(\varphi_0)$, $k\in K$ and
$B\in \End (V)$, we have \begin{equation*} s(B)(\Ad^{\ast}
(k)\varphi)=s(\rho (k)^{-1}B \rho (k))(\varphi).\end{equation*}

\item  For each $A\in {\mathfrak k}$ and $\varphi \in o(\varphi_0)$,
we have $s(d\rho(A))(\varphi)=i\langle \varphi, A\rangle $.
\end{enumerate}
\end{proposition}

In our papers \cite{CaBeit4}, \cite{CaCmuc3} and \cite{CaRimut}, we developped a
general method for constructing a Berezin quantization map associated with a unitary
representation of a quasi-Hermitian Lie group which is holomorphically induced from
a unitary irreducible representation of a maximal compactly embedded subgroup.
This construction goes as follows. 

The evaluation maps $K_z:{\mathcal F}\rightarrow V,
f\rightarrow f(z)$ are continuous \cite{Ne}, p. 539. The vector
coherent states of $\mathcal F$ are the maps
$E_z=K^{\ast}_z:V\rightarrow {\mathcal F}$ defined by $\langle
f(z),v\rangle_V=\langle f,E_zv\rangle_{\mathcal F}$ for $f\in {\mathcal F}$ and
$v\in V$. Here we have that $E_zv=e_z\otimes v$, that is, we have
$(E_zv)(w)=e^{ {\bar z}w/2\lambda}v$.

 Let ${\mathcal F}^s$ be the subspace of $\mathcal F$ generated
by the functions $e_z\otimes v$ for $z\in {\mathbb C}^n$ and $v\in V$.
Then ${\mathcal F}^s$ is a dense subspace of $\mathcal F$. Let
$\mathcal C$ be the space consisting of all operators $A$ on
$\mathcal F$ such that the domain of $A$ contains ${\mathcal F}^s$
and the domain of $A^{\ast}$ also contains ${\mathcal F}^s$. 
Then, following an idea of \cite{Kil} and \cite{AE2}, we first introduce
the pre-symbol $S_0(A)$ of $A\in {\mathcal C}$ by
\begin{equation*}
S_0(A)(z)=(E_z^{\ast}E_z)^{-1/2}E_z^{\ast}AE_z(E_z^{\ast}E_z)^{-1/2}=
e^{- z {\bar z}/2\lambda}
E_z^{\ast}AE_z.
\end{equation*}

The Berezin symbol $S(A)$ of $A$ is thus defined as the
complex-valued function on ${\mathbb C}^n\times o(\varphi_0)$ given
by
\begin{equation*}S(A)(z,\varphi)=s(S_0(A)(z))(\varphi).\end{equation*}

By applying Proposition 4.4 of \cite{CaRimut} we can see that $S$ has the following properties.

\begin{proposition} \label{proprigrand} \begin{enumerate}
\item Each $A\in \mathcal C$ is determined by $S(A)$.

\item For each $A\in {\mathcal C}$, we
have $S(A^{\ast})=\overline {S(A)}$.

\item We have $S(I_{\mathcal F})=1$.

\item For each $A\in {\mathcal C}$, $g=((z_0,{\bar
z_0}),c,k)\in G$, $z\in {\mathbb C}^n$
and $\varphi \in o(\varphi_0)$, we have
\begin{equation*}S(A)(g\cdot z,\varphi)=S(\pi (g)^{-1}A \pi (g))(z,
\Ad^{\ast} (k^{-1})\varphi).\end{equation*} 
\end{enumerate}\end{proposition}

Moreover, we can decompose $S$ according to the decomposition ${\mathcal F}={\mathcal F}_0\otimes V$.
Let $f_0$ be a complex-valued function on ${\mathbb C}^n$ and $f_1$ be a complex-valued function on
$o(\varphi_0)$. Then we denote by $f_0\otimes f_1$ the function on ${\mathbb C}^n\times o(\varphi_0)$ defined by
$(f_0\otimes f_1)(z,\varphi)=f_0(z)f_1(\varphi)$.

\begin{proposition} \label{propdecomp} Let $A_0\in {\mathcal C}_0$ and let $A_1$ be an operator on $V$. Then $A_0\otimes A_1\in {\mathcal C}$ and we have
$S(A_0\otimes A_1)=S^0(A_0)\otimes s(A_1)$. \end{proposition} 

From this, we deduce the following result. We denote by $\varphi^0$ the restriction to $\mathfrak g$
of the extension of $\tilde{\varphi_0}\in ({\mathfrak k}^c)^{\ast}$ to ${\mathfrak g}^c$ which vanishes
on ${\mathfrak p}^{\pm}$. We also denote by ${\mathcal O}(\varphi^0)$ the orbit of $\varphi^0$ for the coadjoint
action of $G$.

\begin{proposition} \label{symbpi} \cite{CaRimut} \begin{enumerate} 
\item Let $g=((z_0,{\bar z_0}),c_0,k)\in G$. For each $z\in {\mathbb C}^n$ and 
$\varphi\in o(\varphi_0)$, we have 
\begin{equation*}S(\pi(g))(z,\varphi)=\exp \left(i \lambda c_0+\tfrac{1}{2}i{\bar z_0}z-\tfrac{\lambda}{4}\vert z_0\vert^2-\tfrac{1}{2\lambda}\vert z\vert^2+\tfrac{1}{2\lambda}{\bar z}k^{-1} (z+i\lambda z_0)\right)\,s(\rho(k))(\varphi).
\end{equation*} 
 \item For each $X=((a,{\bar a}),c,A)\in {\mathfrak g}$,
$z\in {\mathbb C}^n$ and 
$\varphi\in o(\varphi_0)$,  we have
\begin{equation*} S(d\pi (X))(z,\varphi)=i\lambda c+\frac{i}{2}\left( {\bar a}z+ a{\bar z}\right)
-\frac{1}{2\lambda}{\bar z}(Az)
+s(d\rho(A))(\varphi).\end{equation*}
\item For each $X=((a,{\bar a}),c,A)\in {\mathfrak g}$,
$z\in {\mathbb C}^n$ and 
$\varphi\in o(\varphi_0)$,  we have
$$S(d\pi (X))(z,\varphi)=i\langle \Phi (z,\varphi), X \rangle$$ where the map $\Phi: {\mathbb C}^n\times o(\varphi_0)\rightarrow
{\mathfrak g}^{\ast}$ is defined by
\begin{equation*}\Phi (z,\varphi)=\left(i (-z,{\bar z}), \lambda, \varphi-\tfrac{1}{2\lambda}(z,{\bar z}) 
\times (z,{\bar z})\right).\end{equation*}
Moreover $\Phi$ is a diffeomorphism from ${\mathbb C}^n\times o(\varphi_0)$ onto ${\mathcal O}(\varphi^0)$. \end{enumerate} \end{proposition}

Consider now the Berezin transforms ${\mathcal B}:=SS^{\ast}$, ${\mathcal B}^0:=S^0(S^0)^{\ast}$,
$b:=ss^{\ast}$ and the corresponding maps $U:={\mathcal B}^{-1/2}S$, $U^0:=({\mathcal B}^0)^{-1/2}S^0$ and $w:=b^{-1/2}s$. We fix a $K$-invariant measure $\nu$ on
$o(\varphi_0)$ and we endow  ${\mathbb C}^n \times o(\varphi_0)$ with the measure $\mu_{\lambda}\otimes \nu$.
Also, we consider the action of $G$ on  ${\mathbb C}^n \times o(\varphi_0)$ given by 
\begin{equation*}g\cdot (z,\varphi):=(g\cdot z, \Ad^{\ast}(k)\varphi)\end{equation*}
for $g=((z_0,{\bar z_0}),c_0,k)\in G$. Then we have the following results.

\begin{proposition} \label{UisSW} \cite{CaRimut} The map $U$ is a Stratonovich-Weyl correspondence for
$(G,\pi, {\mathbb C}^n \times o(\varphi_0))$. \end{proposition}

\begin{proposition} \label{decomptb} \cite{CaRimut}  For each $f\in L^2({\mathbb C}^n \times o(\varphi_0), \mu_{\lambda}\otimes \nu)$, we have 

\begin{equation*}
{\mathcal B}(f)(z,\psi)=\int_{{\mathbb C}^n \times o(\varphi_0)} k_{\mathcal B}(z,w,\psi,
\varphi)\,f(w,\varphi)\,d\mu_{\lambda} (w)d\nu(\varphi)
\end{equation*} where
\begin{equation*} k_{\mathcal B}(z,w,\psi, \varphi):=e^{-\vert z-w\vert^2/2\lambda}\,
\frac {\vert \langle e_{\hat {\psi}},e_{\hat{\varphi}}\rangle_V\vert^2 } { \langle
e_{\hat{\varphi}},e_{\hat{\varphi}}\rangle_V\langle e_{\hat {\psi}},e_{\hat {\psi}}\rangle_V
}.\end{equation*}

In particular, for each $f_0\in  L^2({\mathbb C}^n)$ and $f_1\in Sy(o(\varphi_0))$, we have
${\mathcal B}(f_0\otimes f_1)={\mathcal B}^0(f_0)\otimes b(f_1)$.  Moreover for each $A_0$ operator on ${\mathcal F}_0$
and $A_1$ operator on $V$, we have $U(A_0\otimes A_1)=U^0(A_0)\otimes w(A_1)$.
\end{proposition}

Note that it is well-known that if ${\Delta}:=4\sum_{k=1}^n  (\partial_{z_k}  \partial_{\bar z_k} )$ is the Laplace operator then we have ${\mathcal B}^0=\exp (\lambda{\Delta}/2)$, see \cite{Luo}. Thus we get 
$U^0=\exp (-\lambda{\Delta}/4)S^0$. Hence, by applying Proposition \ref{symbpi} and Proposition \ref{decomptb}, we obtain the following result.

\begin{proposition} \label{propswdpi} \cite{CaRimut}  For each $X=((a,{\bar a}),c, A)\in {\mathfrak g}$, $z\in {\mathbb C}^n$ and $\varphi \in o(\varphi_0)$, we have
\begin{equation*} U(d\pi(X))(z,\varphi)= ic\lambda +w(d\rho(A))(\varphi)+\frac{1}{2}\Tr (A)+
\frac{i}{2}\left( {\bar a}z+a{\bar z}\right)
-\frac{1}{2\lambda}{\bar z}(Az).\end{equation*}\end{proposition}

\section {Schr\"odinger model for Heisenberg motion groups} \label{sec:Sch}

Here we introduce the Schr\"odinger representations of $G$ by using a Segal-Bargmann transform
which is obtained by a slight modification of $B_0$. More precisely, let us define the map $B$ from  $L^2({\mathbb R}^n,V)\cong  L^2({\mathbb R}^n)\otimes V$ to ${\mathcal F}\cong {\mathcal F}_0 \otimes V$ by $B:=B_0\otimes I_V$ or, equivalently, by the integral formula

\begin{equation*}B(f)(z)=(\lambda/ \pi)^{n/4}\,\int_{{\mathbb R}^n}\,e^{
(1/4\lambda)z^2+ixz-(\lambda/2)x^2}\,f(x)\,dx\end{equation*}
for each $f\in L^2({\mathbb R}^n,V)$.

Now, by analogy with the case of the Heisenberg group, we define the Schr\"odinger representation $\sigma$ of $G$
on $L^2({\mathbb R}^n,V)$ by $\sigma (g):=B^{-1}\pi(g)B$. Similarly, recalling that $\tau$ is the representation of $K$ on ${\mathcal F}_0$ given by $(\tau(k)F)(z)=F(k^{-1}z)$, we define the representation $\tilde{\tau}$ of $K$ on $ L^2({\mathbb R}^n,V)$ by $\tilde{\tau}:=B_0^{-1}\tau(k)B_0$. Then we have the following proposition.

\begin{proposition} \label{propdecsigma} Let $g_0\in G_0$, $k\in K$ and $g=(g_0,k)\in G$. Then we have
$\sigma (g)=\sigma_0 (g_0)\tilde{\tau}(k)\otimes \rho(k)$. \end{proposition}

\begin{proof} Let $f_0\in L^2({\mathbb R}^n)$ and $v\in V$. Then by Eq. \ref{eq:decomppi} we have
\begin{align*}\sigma (g)(f_0\otimes v)&=(B_0^{-1}\otimes I_V)(\pi_0(g_0)\tau(k)\otimes \rho(k)) (B_0\otimes I_V)(f_0\otimes v)\\
&=(B_0^{-1}\pi_0(g_0)\tau(k)B_0)f_0\otimes \rho(k)v \\
&=\sigma_0 (g_0)(B_0^{-1}\tau(k)B_0)f_0 \otimes \rho(k)v,\\
\end{align*}
hence the result. \end{proof}

The following proposition gives an explicit expression for $d\sigma(X)$ when $X$ is of the form $((0,0),0,A)$
where $A\in {\mathfrak k}$.

\begin{proposition} \label{propdtau} \begin{enumerate}
\item For each $A=(a_{kl})\in {\mathfrak k}$, we have 
\begin{equation*}d\tilde{\tau}(A)=\frac{1}{2\lambda}\sum_{k,l}a_{kl}\frac{\partial^2}{\partial x_k\partial x_l}
+\frac{1}{2}\sum_{k,l}a_{kl}\left(x_k\frac{\partial}{\partial x_l}-x_l\frac{\partial}{\partial x_k}\right)
-\frac{\lambda}{2}x(Ax)+\frac{1}{2}\Tr (A). \end{equation*} 
\item For each $X=((0,0),0,A)$ with $A\in {\mathfrak k}$, we have
\begin{equation*}d\sigma(X)=d\tilde{\tau}(A)\otimes I_V+I_{{\mathcal F}_0}\otimes d\rho(A)\end{equation*}
where $d\tilde{\tau}(A)$ is as in (1).
\end{enumerate}\end{proposition}

\begin{proof} In order to prove the first statement, first note that for each $A\in {\mathfrak k}$ and $F^0\in {\mathcal F}_0$ we have
\begin{equation*}(d\tau(A)F^0)(z)=-(dF^0)_z(Az)=-\sum_k \frac{\partial F^0}{\partial z_k}(z)(e_k(Az)).\end{equation*}

To simplify the notation we denote by $k_{B_0}(z,x)$ the kernel of $B_0$, that is, 
\begin{equation*}k_{B_0}(z,x):=\left(\frac{\lambda}{\pi}\right)^{n/4}e^{
(1/4\lambda){ z}^2+ix{ z}-(\lambda/2)x^2}.\end{equation*}

Then, for each $f_0\in {\mathcal S}({\mathbb R}^n)$ we have
\begin{equation*} (d\tau (A)B_0f_0)(z)=-\int_{{\mathbb R}^n}\, \left(\frac{1}{2\lambda}z(Az)+ix(Az)\right)k_{B_0}(z,x)f_0(x)dx. \end{equation*}

Thus writing $z(Az)=\sum_{k,l}a_{kl}z_kz_l$ and integrating by parts, we get
\begin{equation*}\int_{{\mathbb R}^n}\, z(Az) k_{B_0}(z,x)f_0(x)dx=-\left(\frac{\lambda}{\pi}\right)^{n/4}
\sum_{k,l}a_{kl}\int_{{\mathbb R}^n}\,e^{
(1/4\lambda){z}^2+ix{z}}\frac{\partial^2}{\partial x_k\partial x_l}(e^{-(\lambda/2)x^2}f_0(x))dx
\end{equation*}
and, similarly,
\begin{equation*}\int_{{\mathbb R}^n}\, ix(Az) k_{B_0}(z,x)f_0(x)dx
=-\left(\frac{\lambda}{\pi}\right)^{n/4}
\sum_{k,l}a_{kl}\int_{{\mathbb R}^n}\,e^{
(1/4\lambda){z}^2+ix{ z}}\frac{\partial}{\partial x_l}(e^{-(\lambda/2)x^2}x_kf_0(x))dx.
\end{equation*} The first statement hence follows. The second statement is an immediate consequence of Proposition
\ref{propdecsigma} .
\end{proof}

Note that $\sigma$ is completely determined by the fact that $\sigma (g_0, I_n)=\sigma_0(g_0)\otimes I_V$ and by
Proposition \ref{propdtau}.

\section {Stratonovich-Weyl correspondence via Weyl calculus} \label{sec:SWC}

In this section we first introduce a slight modification of the usual Weyl correspondence
in the spirit of our previous works, see for instance \cite{CaWQ}.

Recall that the Berezin calculus $s$ associates with each operator $B$ on $V$
a complex-valued function $s(B)$ on $o({\varphi_0})$ which is called the symbol of $B$ and that the space of all such symbols is denoted by $Sy(o({\varphi_0}))$, see Section \ref{sec:swcb}.
Then the unitary part $w$ of $s$ is an isomorphism
from $\End (V)$ onto $Sy(o({\varphi_0}))$.

Now we say that a complex-valued smooth function $f:\,(p,q,\varphi)\rightarrow
f(p,q,\varphi)$ is a symbol on $ {\mathbb R}^{2n}\times
o(\varphi_0)$ if for each $(p,q)\in{\mathbb R}^{2n}$ the function
$f(p,q,\cdot):\varphi \rightarrow f(p,q,\varphi)$ is an element of  $Sy(o({\varphi_0}))$.
In that case, we denote ${\hat f}(p,q):=w^{-1}(f(p,q,\cdot))$.
A symbol $f$ on ${\mathbb R}^{2n}\times
o(\varphi_0)$
is called an {\it S}-symbol if the function $\hat {f}$ belongs to
the Schwartz space ${\mathcal S}(\mathbb R^{2n},\End(V))$ of rapidly decreasing smooth functions on ${\mathbb R}^{2n}$ with values in $\End (V)$. For each
{\it S}-symbol on ${\mathbb R}^{2n}\times
o(\varphi_0)$, we define
the operator $ W(f)$ on the Hilbert space $L^2(\mathbb R^n,V)$ by
 \begin{equation*}W(f)\phi
(p)={(2\pi)}^{-n}\,\int_{{\mathbb R}^{2n}}\, e^{isq} \,{\hat f}( p+(1/2)s,
q)\,\phi (p+s)\,ds\,dq.
\end{equation*}

Of course, $W$ can be extended to much larger
classes of symbols as the usual Weyl calculus, see Section \ref{sec:Heisen}.
As an immediate consequence of the definition of $W$, we have the following proposition.

\begin{proposition} \label{propDecompW} \begin{enumerate}\item $W$ is a unitary operator from $L^2(\mathbb R^{2n},V)$ onto ${\mathcal L}_2(L^2(\mathbb R^{n},V))$;
\item For each $f_0\in {\mathcal S}(\mathbb R^n)$ and $f_1\in Sy(o({\varphi_0}))$, we have
$W(f_0\otimes f_1)=W_0(f_0)\otimes w^{-1}(f_1)$.
\end{enumerate}\end{proposition}

In order to compare $W$ and $U$, it is convenient to transfer $U$ to operators on $L^2(\mathbb R^n,V)$
in the spirit of Proposition \ref{propLuo}. First, for any operator $A$ on $L^2(\mathbb R^n,V)$, we define
$S_1(A):=S(BAB^{-1})$. Clearly, one has $S_1S_1^{\ast}=SS^{\ast}={\mathcal B}$. Then the unitary part
$U_1$ of $S_1$ is given by $U_1(A):=U(BAB^{-1})$ for any operator $A$ on $L^2(\mathbb R^n,V)$.
Moreover, we have
\begin{equation}\label{eq:decU1} U_1={\mathcal B}^{-1/2}S_1=
\left(({\mathcal B^0})^{-1/2}\otimes b^{-1/2}\right)\left(S^1\otimes s\right)
=({\mathcal B^0})^{-1/2}S^1\otimes b^{-1/2}s=U^1\otimes w
\end{equation} 
with obvious notation. Hence we are in position to extend Proposition  \ref{propLuo} to Heisenberg motion groups.

\begin{proposition} \label{propLuobis} We have $U_1=(J^{-1}\otimes I_{Sy(o({\varphi_0}))})W^{-1}$.\end{proposition}

\begin{proof} By Proposition \ref{propDecompW}, Proposition \ref{propLuo} and Eq. \ref{eq:decU1},
we have 
\begin{equation*}(J^{-1}\otimes I_{Sy(o({\varphi_0}))})W^{-1}=(J^{-1}\otimes I_{Sy(o({\varphi_0}))})(W_0^{-1}\otimes w)=(J^{-1}W_0^{-1})\otimes w=U^1\otimes w=U_1.\end{equation*} This is the desired result. \end{proof}

Now consider the action of $G$ on $\mathbb R^{2n} \times o(\varphi_0)$ given by
\begin{equation*}g\cdot (p,q,\varphi):=(j^{-1}(g\cdot j(p,q)), \Ad^{\ast}(k)\varphi)\end{equation*}
for $g=((z_0,{\bar z_0}),c_0,k)\in G$. Then we have the following result.

\begin{proposition}\label{corol} \begin{enumerate} \item $W^{-1}$ is a Stratonovich-Weyl correspondence for $(G, \sigma, \mathbb R^{2n} \times o(\varphi_0))$.
\item  For each $X=((a,{\bar a}),c,A)\in {\mathfrak g}$,
$p,q\in {\mathbb R}^n$ and 
$\varphi\in o(\varphi_0)$,  we have
\begin{align*} W^{-1}(d\sigma(X))&(p,q,\varphi)=i\lambda c+\frac{1}{2}\Tr (A)+\frac{i}{2}\left( {\bar a} j(p,q)+ a\overline{ j(p,q)}\right)\\
&-\frac{1}{2\lambda}{\overline { j(p,q)}}(A{ j(p,q)})
+w(d\rho(A))(\varphi).\end{align*}\end{enumerate}\end{proposition}

\begin{proof} (1) For each $g=((z_0,{\bar z_0}),c_0,k)\in G$ let us denote by $L_g$ the operator of $L^2({\mathbb C}^n\times o(\varphi_0), \mu_{\lambda}\otimes \nu)$ defined by 
\begin{equation*}(L_gF)(z,\varphi)=F(g\cdot z,\Ad^{\ast}(k)\varphi).\end{equation*}
Then the covariance property for $U$ can be rewritten as
\begin{equation*}L_gU(A)=U(\pi(g)^{-1}A\pi(g))\end{equation*}
for each $g\in G$ and $A\in {\mathcal L}_2(\mathcal F)$. This gives the following covariance property for $U_1$:
\begin{equation*}L_gU_1(A)=U_1(\sigma(g)^{-1}A\sigma(g))\end{equation*}
for each $g\in G$ and $A\in {\mathcal L}_2(L^2({\mathbb R}^n,V))$. But by Proposition \ref{propLuobis} we have
 $U_1=(J^{-1}\otimes I_{Sy(o({\varphi_0}))})W^{-1}$. Thus we get 
\begin{equation*}(J\otimes I_{Sy(o({\varphi_0}))})L_g(J^{-1}\otimes I_{Sy(o({\varphi_0}))})W^{-1}(A)=W^{-1}(
\sigma(g)^{-1}A\sigma(g)) \end{equation*} for each $g\in G$ and $A\in {\mathcal L}_2(L^2({\mathbb R}^n,V))$.

Now let \begin{equation*}({\tilde L}_gf)(p,q,\varphi):=f(j^{-1}(g\cdot j(p,q)), \Ad^{\ast}(k)\varphi) \end{equation*}
for each $g=((z_0,{\bar z_0}),c_0,k)\in G$ and $ (p,q,\varphi)\in \mathbb R^{2n} \times o(\varphi_0)$. Since it is clear that for each $g\in G$ we have 
\begin{equation*}{\tilde L}_g=(J\otimes I_{Sy(o({\varphi_0}))})L_g(J^{-1}\otimes I_{Sy(o({\varphi_0}))}),\end{equation*}
we see that 
\begin{equation*}{\tilde L}_gW^{-1}(A)=W^{-1}(
\sigma(g)^{-1}A\sigma(g)) \end{equation*} for each $g\in G$ and $A\in {\mathcal L}_2(L^2({\mathbb R}^n,V))$.
Hence $W^{-1}$ is $G$-covariant. The other properties of a Stratonovich-Weyl correspondence can be easily verified.

(2) For each $X\in {\mathfrak g}^c$, we have
\begin{equation*}U(d\pi(X))=U_1(d\sigma(X))=((J^{-1}\otimes I_{Sy(o({\varphi_0}))})W^{-1}(d\sigma(X))
\end{equation*} hence the result follows from Proposition \ref{propswdpi}.
\end{proof}

Finally, we can obtain Stratonovich-Weyl correspondences for $(G, \pi,{\mathcal
O}({\varphi^0}))$ and for $(G,
\sigma,{\mathcal O}({\varphi^0}))$ by transferring $U$ and $W^{-1}$ 
by means of $\Phi$. Let $$\Psi:=\Phi \circ (j\otimes 1): {\mathbb R}^{2n}\times o(\varphi_0)\rightarrow {\mathcal O}({\varphi^0})$$ and let $\tilde{\nu}$ be the $G$-invariant measure on ${\mathcal O}({\varphi^0})$ defined by
$\tilde {\nu}:=(\Phi^{-1})^{\ast}(\mu_{\lambda}\otimes \nu)=(\Psi^{-1})^{\ast}({\tilde
\mu}\otimes \nu)$. Consider also the unitary maps 
$\tau_{\Phi}:F\rightarrow F\circ
\Phi^{-1}$ from $L^2({\mathbb C}^n\times o(\varphi_0) ,\mu_{\lambda}\otimes \nu)$ onto
$L^2({\mathcal O}({\varphi^0}),\tilde {\nu})$ and
$\tau_{\Psi}:F\rightarrow F\circ \Psi^{-1}$ from
$L^2({\mathbb R}^{2n}\times o(\varphi_0),{\tilde \mu}\otimes \nu)$ onto $L^2({\mathcal
O}(\varphi^0),\tilde {\nu})$. Then we have the following proposition.

\begin{proposition} \label{finalswc} The map ${\mathcal W}'_1:=\tau_{\Psi}W^{-1}$ is a
Stratonovich-Weyl correspondence for $(G, \sigma,{\mathcal
O}(\varphi^0))$, the map ${\mathcal W}'_2:=\tau_{\Phi}U$
is a Stratonovich-Weyl correspondence for $(G,
\pi,{\mathcal O}(\varphi^0))$ and we have ${\mathcal
W}'_1={\mathcal W}'_2I_B$.
\end{proposition}

\begin{proof} The first and the second assertions immediately follow from Proposition \ref{UisSW} and Proposition \ref{corol}. To prove the third assertion, note that we have $\tau_{\Psi}(J\otimes  I_{Sy(o({\varphi_0}))})=\tau_{\Phi}$. Then, by Proposition \ref{propLuobis}, we can write
\begin{equation*}{\mathcal W}'_2I_B=\tau_{\Phi}UI_B=\tau_{\Phi}U_1=\tau_{\Phi}(J^{-1}\otimes  I_{Sy(o({\varphi_0}))})W^{-1}=\tau_{\Psi}W^{-1}={\mathcal W}'_1,\end{equation*} hence the result. \end{proof}

\end{document}